\documentclass[a4paper]{amsart}
\usepackage[utf8]{inputenc}

\usepackage[T1]{fontenc}

\usepackage{amsmath,amssymb,mathtools}
\usepackage[left=2.5cm,right=2.5cm,top=2cm,bottom=2cm]{geometry}
\usepackage[
    pagebackref,
    colorlinks=true,
    citecolor=red,
]{hyperref}
\newtheorem{theorem}{Theorem}[section]
\newtheorem{lemma}[theorem]{Lemma}

\newtheorem{proposition}[theorem]{Proposition}

\theoremstyle{definition}

\theoremstyle{remark}
\newtheorem{remark}[theorem]{Remark}

\numberwithin{equation}{section}

\newcommand*{\RR}{\mathbb{R}}
\newcommand*{\NN}{\mathbb{N}}
\newcommand*{\calE}{\mathcal{E}}
\newcommand*{\calX}{\mathcal{X}}

\DeclareMathOperator{\EE}{\mathbb{E}} 
\DeclareMathOperator{\PP}{\mathbb{P}} 
\DeclareMathOperator{\Ent}{Ent} 
\DeclareMathOperator{\Var}{Var} 

\DeclarePairedDelimiter{\abs}{\vert}{\vert}

\newcommand*{\indbr}[1]{{\bf 1}_{ \{ #1 \} }}

\title{Concentration bounds for sampling without replacement and Hoeffding statistics}
\author{Bart{\l}omiej Polaczyk}
\address{Institute of Mathematics, University of Warsaw, Poland}
\email{B.Polaczyk@mimuw.edu.pl}
\thanks{Research partially supported by the National Science Centre, Poland, via the Preludium grant no.\ 2020/37/N/ST1/02667.}

\begin{document}

\begin{abstract}
    We prove a Bennett-type concentration bound for suprema of empirical processes based on sampling without replacement and a corresponding bound in the case of an arbitrary Hoeffding statistics.
    We improve on the previous results of such type, providing a sharper concentration profile.
\end{abstract}
\maketitle



\section{Preliminaries}
In this short note we investigate concentration properties of particular functionals of uniform random permutations.
Namely, we focus on the suprema of empirical processes when sampling without replacement.
Such processes can be seen as Hoeffding statistics for matrices of a special form with repeated rows.
We also obtain corresponding bounds for a single Hoeffding statistics for general underlying matrix.
Such bounds were considered extensively in the literature, cf., e,g,~\cite{MR3363542,MR3185193,MR1419006}, and they play an important role in various applications, e.g., in transductive learning~\cite{tolstikhin2014localized}, or statistical testing~\cite{mlis2018concentration}.

\subsection{Organization of this paper}
In the rest of this section we introduce some core notation.
In Section~\ref{sec:suprema} we present our results concerning concentration for suprema of empirical processes when sampling without replacement.
In Section~\ref{sec:single-hoeffding} we present analogous results for a single Hoeffding statistic.
We provide remaining proofs of our concentration estimates in Section~\ref{sec:missing-proofs}.
Proofs of auxiliary facts and some additional discussion is moved to Appendix.
\subsection{Basic notation}
For $n\in\NN$, consider the symmetric group $S_n$ of permutations of the set $[n]:=\{1,\ldots,n\}$ equipped with the uniform probability measure $\pi_n$. 
It is the stationary distribution of the \emph{interchange process} defined via its generator $L$ given by the formula
\begin{displaymath}
  L f(\sigma) = 
  \frac{1}{n(n-1)}
  \sum_{i,j = 1}^n 
  \big(f(\sigma \circ \tau_{ij}) - f(\sigma)\big) = 
  \frac{2}{n(n-1)}
  \sum_{1\le i < j \le n}
  \big(f(\sigma \circ \tau_{ij}) - f(\sigma)\big),
\end{displaymath}
where $\tau_{ij}$ stands for the transposition of elements $i$ and $j$.
By $\EE$, we denote the expectation w.r.t. $\pi_n$.
Moreover, for a function $f\colon S_n\to\RR$, denote 
$f_{ij}(\cdot) = f(\cdot\circ\tau_{ij})$ 
for short.
The corresponding Dirichlet form is then expressed as
\begin{align*}
    \calE (f,g) &= 
    \frac{1}{2n(n-1)}
    \EE
    \sum_{i,j=1}^n 
    (f_{ij} - f)
    (g_{ij} - g)
    \\& = 
    \frac{1}{n(n-1)}
    \EE
    \sum_{1\le i < j\le n}
    (g_{ij} - g)
    (f_{ij} - f).
\end{align*}
If $f$ and $g$ have the same monotonicity, then by the reversibility of $L$ we also have
\[
    \calE(f, g)
    =
    \frac{1}{n(n-1)}
    \EE
    \sum_{i,j=1}^n
    (g_{ij} - g)_+
    (f_{ij} - f)_+.
\]

We say that the \emph{modified log-Sobolev} inequality is satisfied with constant $\rho_0>0$ if
\begin{equation}\label{eq:mLS}
    \rho_0
    \Ent_\mu(f)
    \le 
    \calE(f, \log f)
\end{equation}
for all positive functions $f$, where $\Ent_\mu(f) = \int f\log f\,d\mu - \int f\,d\mu \log (\int f\,d\mu)$ is the entropy functional.
For this process, $\rho_0 \ge \frac{1}{n-1}$ was obtained independently by Gao--Quastel~\cite{MR2023890} and Bobkov--Tetali~\cite{MR2283379} (note that the normalization of the generator $L$ differs across various references -- we provide here scaled constants matching our setting).

\section{Sampling without replacement -- concentration for suprema}\label{sec:suprema}
Consider a set of vectors $\calX \subset \RR^n$.
Let $I_1,\ldots, I_n$ be a uniform sample without replacement and $J_1,\ldots, J_n$ be a sample with replacement from the set $[n]$.
For $m\le n$, define 
\begin{equation}\label{eq:Z_def}
    Z = 
    \sup_{x\in\calX} \sum_{k=1}^m x_{I_k},
    \qquad
    Z' = 
    \sup_{x\in\calX} \sum_{k=1}^m x_{J_k}
\end{equation}
so that $Z'$ can be considered a supremum of the empirical process in independent random variables $J_k$.
Tails of $Z'$ have been extensively studied beginning with the work of Talagrand~\cite{MR1419006}.

To analyze the tails of $Z$, it is often convenient to represent it as a supremum of Hoeffding statistics over a family of matrices.
Namely, for $x\in\calX$, denote $a^x\in\RR^{n\times n}$ to be such that the first $m$ rows of $a$ consist of copies of vector $x$ and the remaining rows have zero entries only, i.e., $a_{ij}=x_j$ for $i\le m$, $j\in[n]$ and $a_{ij}=0$ for $i>m$, $j\in[n]$.
Then
\[
  Z = \sup_{x\in\calX}\sum_{k=1}^n a^x_{k\sigma_k},
\]
where $\sigma = (I_1, I_2,\ldots, I_n) \sim \pi_n$.
Moreover, denote $\sigma_{ij}=\sigma\circ \tau_{ij}$ for any $i,j\in[n]$ and
\[
    Z_{ij} = 
    \sup_{x\in\calX}
    \sum_{k=1}^n 
    a^x_{k\sigma_{ij}(k)},
\]
so that the modified log-Sobolev inequality~\eqref{eq:mLS} applied to the Laplace transform of $Z$ reads 
\begin{equation*}
    \Ent(e^{\lambda Z})
    \le 
    \frac{\lambda}{n}
    \EE 
    e^{\lambda Z}
    \sum_{ij}(1-e^{-\lambda(Z-Z_{ij})})_+(Z-Z_{ij})_+.
\end{equation*}
In the sequel, we express our concentration results for $Z$ using the following quantities
\begin{equation*}
    \Sigma^2 
    = 
    \sup_{x\in\calX}
    \sum_{k=1}^m x_{I_k}^2,
    \qquad 
    \widetilde{\Sigma}^2 
    = 
    \sup_{x\in\calX}
    \sum_{k=1}^m x_{J_k}^2.
\end{equation*}
As pointed out in~\cite{gross2010note}, it follows from an argument due to 
Hoeffding~\cite{MR144363} (cf. also~\cite{MR3193733}) that if $E$ is a normed space and 
$g\colon [n]\to E$, then for any convex function
$\Psi\colon E\to \RR$,
\begin{equation}\label{eq:Hoeffding_argument}
    \EE \Psi\Bigl(
        \sum_{k=1}^m g({I_k})
    \Bigr)
    \le
    \EE \Psi\Bigl(
        \sum_{k=1}^m g({J_k})
    \Bigr).
\end{equation}
The meaning of~\eqref{eq:Hoeffding_argument} in terms of $Z$ and $Z'$ and related quantities is explained in the following lemma, which in particular implies that 
$\EE Z \le \EE Z'$ and
$\EE \Sigma^2 \le \EE \widetilde{\Sigma}^2$.
We provide its proof for completeness in Appendix~\ref{sec:hoeffding}.
\begin{lemma}\label{L:hoeffding}
 Let $\phi\colon\RR\to\RR$ be convex and increasing, and let $Z,Z'$ be given by~\eqref{eq:Z_def}.
 Then 
 \[
    \EE \phi(Z) \le \EE \phi(Z').
 \]
\end{lemma}

Our main result regarding concentration of $Z$ is the theorem below providing a Bennett-type bound.
\begin{theorem}\label{T:Z-Bennett-general}
    Let $Z$ be given by~\eqref{eq:Z_def} and assume $\calX\subset[-1,1]^n$.
    Then, for some absolute constants $C_1,C_2>0$,
    \[
        \forall\; t\ge 0
        \qquad 
        \PP(Z \ge \EE Z + t)
        \le 
        2\exp\Bigl(
            -\frac{t}{C_1}
            \log\Bigl(
                1 + \frac{t}{C_2\EE \widetilde{\Sigma}^2}
            \Bigr)
        \Bigr),
    \]
    where
    $\widetilde{\Sigma}^2 = \sup_{x\in\calX} \sum_{k=1}^m x_{J_k}^2$.
    One can take $C_1=36$, $C_2=46$.
\end{theorem}
\begin{remark}\label{R:TBK-comparision}
    Assume that 
    $
        \calX 
        \subset 
        \{\, 
            x\in [-1,1]^n\colon \sum_i x_i = 0 
        \,\}
    $
    and denote $v=m \sup_{x\in\calX}\Var(x_{J_1}) + 2\EE Z'$.
    Then, Tolstikhin--Blanchard--Kloft~\cite[Theorem 2]{tolstikhin2014localized} proved that
    \begin{equation}\label{eq:Tolstikhin}
        \forall\;t\ge0\qquad
        \PP(Z \ge \EE Z' + t)
        \le 
        \exp\Bigl(
            -
            t
            \log\bigl( 
                1+\frac{t}{v}
            \bigr)
            +
            t 
            -
            v
            \log\bigl( 
                1+\frac{t}{v}
            \bigr)
        \Bigr).
    \end{equation}
    Recall that by Hoeffding's argument~\eqref{eq:Hoeffding_argument}, cf. Lemma~\ref{L:hoeffding}, $\EE Z \le \EE Z'$ and in many situations the latter quantity can be significantly larger.
    Using symmetrization and Talagrand's contraction principle for Rademacher averages, cf., e.g.,~\cite{MR1102015}, we can estimate 
    \[
        \EE\widetilde{\Sigma}^2 
        \le 
        m\sup_{x\in\calX} \Var(x_{J_1}) + 
        8 \EE \sup_{x\in\calX} \sum_{k=1}^m \varepsilon_k x_{J_k},
    \]
    where $\varepsilon_1,\ldots,\varepsilon_m$ are i.i.d. Rademacher variables independent of $J_1,\ldots,J_m$.
    Thus, in the case when the set $\calX$ is symmetric with respect to the origin we obtain that 
    \[ 
        \EE\widetilde{\Sigma}^2  
        \le
        m\sup_{x\in\calX} \Var(x_{J_1}) + 
        16 \EE Z'
        \le 
        8v
    \]
    and consequently our estimate of Theorem~\ref{T:Z-Bennett-general}, in contrast to~\eqref{eq:Tolstikhin}, provides a bound on deviations around the "proper" mean, while having no worse scaling behavior in the exponent (up to numerical constants).

    In the general case however, it does not need to hold that $\EE \widetilde{\Sigma}^2 = \mathcal{O}(v)$, whence the bound~\eqref{eq:Tolstikhin} and our bound of Theorem~\ref{T:Z-Bennett-general} are not directly comparable. 
    It is also worth noting that Authors of~\cite{tolstikhin2014localized} provide a bound $\EE Z' \le \EE Z+2\frac{m^3}{n}$ which shows that one can replace $\EE Z'$ with $\EE Z$ under the probability estimate without losing much for small values of $m$.
    In Appendix~\ref{sec:example}, we provide an example illustrating a situation in which our estimate still improves upon~\eqref{eq:Tolstikhin} in such a general case of non-symmetric set~$\calX$.
\end{remark}
To prove the Bennett-type inequality of Theorem~\ref{T:Z-Bennett-general}, we need the following estimate due to Ledoux~\cite{MR1399224}.
We provide the proof for completeness in Appendix~\ref{sec:Ledoux-pf}.

\begin{lemma}[Proof of Theorem 2.4 in \cite{MR1399224}]\label{T:Z-Bennett-positive}
    Let $Z'$ be given by~\eqref{eq:Z_def} and assume $\calX\subset[0,1]^n$.
    Then 
    \begin{equation*}
        \forall\; \lambda\ge 1/4 
        \qquad 
        \log \EE e^{\lambda Z'}
        \le 
        \frac{1}{16}
        e^{8\lambda}
        \EE Z'.
    \end{equation*}
\end{lemma}
We also need the following proposition providing the Bernstein inequality for $Z$.
We defer its proof to Section~\ref{sec:missing-proofs}.
\begin{proposition}\label{L:Z-Bernstein-general}
    Let $Z$ be given by~\eqref{eq:Z_def} and assume
    $\calX \subset [-1,1]^n$.
    Then
    \[
        \forall\; t\ge 0
        \qquad 
        \PP(Z \ge \EE Z + t)
        \le 
        \exp\Bigl(
            -\min\Bigl(
                \frac{t}{32},
                \frac{t^2}{128 \EE {\Sigma}^2}
            \Bigr)
        \Bigr),
    \]
    where 
    ${\Sigma}^2 = \sup_{x\in\calX} \sum_k x_{I_k}^2$.
\end{proposition}
\begin{proof}[Proof of Theorem~\ref{T:Z-Bennett-general}]
    If 
    $32t < C_1C_2\EE\widetilde{\Sigma}^2$, 
    then we apply Proposition~\ref{L:Z-Bernstein-general} and estimate 
    $\log(1+x)\le x$ 
    to get that as long as 
    $128 \le C_1C_2$,
    \begin{align*}
        \begin{split}
        \PP(Z\ge \EE Z + t)
        &\le 
        \exp\Bigl(
            -\min\Bigl(
                \frac{t}{32},
                \frac{t^2}{128 \EE \widetilde{\Sigma}^2}
            \Bigr)
        \Bigr),
        \\&\le 
        \exp\Bigl(
            -\min\Bigl(
                \frac{t}{32 },
                \frac{t^2}{C_1C_2 \EE \widetilde{\Sigma}^2}
            \Bigr)
        \Bigr)
        \\&=
        \exp\Bigl(
          -\frac{t^2}{C_1C_2\EE\widetilde{\Sigma}^2}
        \Bigr)
        \le 
        \exp\Bigl(
            -\frac{t}{C_1}
            \log\Bigl(
                1 + \frac{t}{C_2\EE \widetilde{\Sigma}^2}
            \Bigr)
        \Bigr)
        \end{split}
    \end{align*}
    and the result follows in this case.

    If $32t \ge C_1C_2\EE\widetilde{\Sigma}^2$, then set 
    \[
        \rho^{-1}
        =
        \alpha
        \log\bigl(
            1+\beta\frac{ t}{\EE \widetilde{\Sigma}^2}
        \bigr)
    \] 
    for some $\alpha, \beta>0$ (to be fixed later)
    and denote 
    \[
        Z^\downarrow
        = 
        \sup_{x\in\calX}
        \sum_{k=1}^m
        x_{I_k}
        \indbr{\abs{x_{I_k}}\le \rho}
    \]  
    and 
    \[
        Z^\uparrow
        =
        \sup_{x\in\calX}
        \sum_{k=1}^m
        \abs{x_{I_k}}
        \indbr{\abs{x_{I_k}} > \rho}
    \]
    so that $Z\le Z^\downarrow + Z^\uparrow$.
    We estimate the tail probabilities for $Z^\downarrow$ and $Z^\uparrow$.

    By the estimate 
    $\log(1+x)\le x$, 
    by the definition of $\rho$ and as long as
    $\alpha\beta \le 1/4$,
    \begin{equation*}
        \frac{t}{32\rho}
        \le 
        \alpha\beta
        \cdot
        \frac{t^2}{32\EE \widetilde{\Sigma}^2}
        \le 
        \frac{t^2}{128\EE \widetilde{\Sigma}^2}
        \le
        \frac{t^2}{128\EE {\Sigma}^2},
    \end{equation*}
    whence, by Proposition~\ref{L:Z-Bernstein-general} applied to $Z^\downarrow/\rho$,
    \begin{align}\label{eq:Z-downarrow-tail}
        \begin{split}
        \PP(Z^\downarrow \ge \EE Z^\downarrow + t)
        &\le 
        \exp\Bigl(
            -\min\Bigl(
                \frac{t}{32\rho},
                \frac{t^2}{128\EE {\Sigma}^2}
            \Bigr)
        \Bigr)
        \\&=
        \exp\Bigl(
            -\frac{t}{32\rho}
        \Bigr)
        =
        \exp\Bigl(
            -
            \frac{\alpha t}{32}
            \log\bigl(
                1+\beta\frac{t}{\EE \widetilde{\Sigma}^2}
            \bigr)
        \Bigr).
        \end{split}
    \end{align}
    
    We turn to the tails of $Z^\uparrow$.
    Denote
    \[
        Z'_\rho
        =
        \sup_{x\in\calX}
        \sum_{k=1}^m
        \abs{x_{J_k}}
        \indbr{\abs{x_{J_k}} > \rho}.
    \]
    Lemma~\ref{L:hoeffding} applied with 
    $\{\, 
        (\abs{x_i}\indbr{\abs{x_i}>\rho})_{i=1}^n
        \;\colon\;
        x\in\calX
    \,\}$ 
    in place of $\calX$
    together with Lemma~\ref{T:Z-Bennett-positive} applied to~$Z'_\rho$ yield
    \begin{align}\label{eq:Z-Bennett-positive-cor}
        \begin{split}
        \log\EE e^{\lambda Z^\uparrow}
        \le  
        \log \EE e^{\lambda Z'_\rho}
        \le
        \frac{1}{16}
        e^{8\lambda}
        \EE Z'_\rho
        \end{split}
    \end{align}
    for all $\lambda\ge 1/4$.
    Choose 
    \begin{equation*}
        \lambda^\ast 
        =
        \frac{1}{8}
        \log\Bigl(
            1+
            \beta
            \frac{t}{\EE \widetilde{\Sigma}^2}
        \Bigr).
    \end{equation*}
    Since $32t\ge C_1C_2\EE\widetilde{\Sigma}^2$ by assumption, 
    then
    $\lambda^\ast \ge \frac{1}{8}\log(1+\frac{\beta C_1C_2}{32}) \ge \frac{1}{4}$,
    as long as $\beta C_1C_2\ge 32(e^2-1)$.
    Moreover, note that 
    \begin{equation*}
        \EE Z_\rho'
        \le 
        \rho^{-1} 
        \EE \widetilde{\Sigma}^2
        \le 
        \frac{32t\rho^{-1}}{C_1C_1}
        .
    \end{equation*}
    Consequently, by the Chernoff bound combined with~\eqref{eq:Z-Bennett-positive-cor},
    \begin{align}\label{eq:Z-uparrow-tail}
    \begin{split}
        \PP(Z^\uparrow \ge t)
        &\le 
        \exp\Bigl(
            -t\lambda^\ast 
            +
            \frac{e^{8\lambda^\ast}}{16}
            \EE Z'_\rho 
        \Bigr)
        \\&=
        \exp\Bigl(
            -\frac{t}{8}
            \log\Bigl(
                1+
                \beta
                \frac{t}{\EE \widetilde{\Sigma}^2}
            \Bigr)
            +
            \frac{1}{16}
            \Bigl(
                \EE Z'_\rho
                +
                \frac{\EE Z'_\rho}{\EE \widetilde{\Sigma}^2}
                t\beta 
            \Bigr)
        \Bigr)
        \\&\le 
        \exp\Bigl(
            -\frac{t}{8}
            \log\Bigl(
                1+
                \beta
                \frac{t}{\EE \widetilde{\Sigma}^2}
            \Bigr)
            +
            \frac{1}{16}
            \Bigl(
                \frac{32t\rho^{-1}}{C_1C_1}
                +
                {t\beta\rho^{-1}}
            \Bigr)
        \Bigr)
        \\&=
        \exp\Bigl(
            -t
            \log\Bigl(
                1+
                \beta
                \frac{t}{\EE \widetilde{\Sigma}^2}
            \Bigr)
            \cdot
            \Bigl(
                \frac{1}{8}
                - 
                \frac{2\alpha}{C_1C_2}
                -
                \frac{\alpha\beta}{16}
            \Bigr)
        \Bigr)
        .
    \end{split}
    \end{align}

    Using the estimate $\log(1+x)\le x$ we obtain that
    \begin{equation}\label{eq:Z-arrow-mean}
        \abs{\EE Z^\downarrow - \EE Z }
        \le 
        \EE Z^\uparrow
        \le 
        \frac{\EE \widetilde{\Sigma}^2}{\rho}
        \le 
        \alpha\beta t.
    \end{equation}
    Thus, combining~\eqref{eq:Z-downarrow-tail},~\eqref{eq:Z-uparrow-tail} and~\eqref{eq:Z-arrow-mean} and as long as 
    $\alpha\beta\le 1/4$
    and 
    $\beta C_1C_2\ge 32(e^2 -1)$,
    we arrive at
    \begin{align*}
        \PP(Z\ge \EE Z + 2t+\alpha\beta t)
        &\le 
        \PP(Z^\uparrow + Z^\downarrow \ge \EE Z + 2t +\alpha\beta t)
        \\&\le 
        \PP(
            Z^\uparrow + Z^\downarrow 
            \ge 
            \EE Z^\downarrow 
            - \vert \EE Z - \EE Z^\downarrow \vert
            + 
            2t+\alpha\beta t
        )
        \\&\le 
        \PP(
            Z^\uparrow + Z^\downarrow 
            \ge 
            \EE Z^\downarrow 
            + 2t
        )
        \\&\le 
        \PP(Z^\uparrow \ge t)
        +
        \PP(Z^\downarrow \ge \EE Z^\downarrow + t)
        \\&\le 
        2
        \exp\Bigl(
            -
            \min\Bigl(
                \frac{\alpha}{32},
                \frac{1}{8}
                -
                \frac{\alpha\beta}{16}
                - 
                \frac{2\alpha}{C_1C_2}
            \Bigr)
            \cdot 
            t
            \log\Bigl(
                1 +
                \beta
                \frac{t}{\EE \widetilde{\Sigma}^2}
            \Bigr)
        \Bigr).
    \end{align*}
    Substituting 
    $t\leftarrow (2+\alpha\beta)^{-1}t$ 
    and estimating 
    $\frac{1}{2+\alpha\beta}\ge \frac{4}{9}$ 
    yields
    \begin{align*}
    \begin{split}
        \PP(Z\ge \EE Z + t)
        &\le 
        2
        \exp\Bigl(
            -
            \frac{1}{2+\alpha\beta}
            \min\Bigl(
                \frac{\alpha}{32},
                \frac{1}{8}
                -
                \frac{\alpha\beta}{16}
                - 
                \frac{2\alpha}{C_1C_2}
            \Bigr)
            \cdot 
            t
            \log\Bigl(
                1 +
                \frac{\beta}{2+\alpha\beta}
                \frac{t}{\EE \widetilde{\Sigma}^2}
            \Bigr)
        \Bigr)
        \\&\le 
        2
        \exp\Bigl(
            -
            \frac{4}{9}
            \min\Bigl(
                \frac{\alpha}{32},
                \frac{1}{8}
                -
                \frac{\alpha\beta}{16}
                - 
                \frac{2\alpha}{C_1C_2}
            \Bigr)
            \cdot 
            t
            \log\Bigl(
                1 +
                \frac{4\beta}{9}
                \frac{t}{\EE \widetilde{\Sigma}^2}
            \Bigr)
        \Bigr)
        \\&\le 
        2
        \exp\Bigl(
            -
            \frac{4}{9}
            \min\Bigl(
                \frac{\alpha}{32},
                \frac{1}{8}
                -
                \frac{\alpha\beta}{16}
                - 
                \frac{2\alpha}{C_1C_2}
            \Bigr)
            \cdot 
            t
            \log\Bigl(
                1 +
                \frac{t}{C_2\EE \widetilde{\Sigma}^2}
            \Bigr)
        \Bigr).
    \end{split}
    \end{align*}
    as long as 
    $\alpha\beta\le 1/4$,
    $\beta C_1C_2\ge 32(e^2 -1)$
    and
    $4\beta C_2\ge 9$.
    Setting 
    $\alpha=2$ and $\beta=\frac{1}{8}$ 
    yields the result with 
    $C_1=36$ and $C_2=46$.
\end{proof}

\section{Concentration for a single Hoeffding statistic}\label{sec:single-hoeffding}
In this section, we provide concentration bounds for single Hoeffding statistics, extending the results of Chatterjee~\cite{MR2288072}, Bercu--Deylon--Rio~\cite{MR3363542} and Albert~\cite{mlis2018concentration}.
In the sequel, $f$ denotes some Hoeffding statistics, i.e.,
\begin{equation}\label{eq:Hoeffding-statistic}
f(\sigma) = \sum_{k=1}^n a_{k\sigma(k)},
\end{equation}
where $(a_{ij})_{i,j=1}^n\in\RR^{n\times n}$ is some real matrix.
The main result of this section is the following theorem.
To the best of our knowledge, this is the first result that captures both the subgaussian and Poisson behaviors of Hoeffding statistics.
\begin{theorem}\label{T:f-Bennett-general}
    Let $f$ be given by~\eqref{eq:Hoeffding-statistic}.
    If $a_{ij} \in [-1,1]$ for all $i,j$ and 
    $\sum_{ij} a_{ij}=0$, then for some absolute constants $C_1,C_2>0$,
    \[
        \forall\; t\ge 0
        \qquad 
        \PP(f\ge t)
        \le 
        2\exp\Bigl(
            -\frac{t}{C_1}
            \log\Bigl(
                1 + \frac{t}{C_2\EE \Sigma^2}
            \Bigr)
        \Bigr),
    \]
    where 
    $\Sigma^2 = \sum_k a_{k\sigma_k}^2$
    so that
    $\EE\Sigma^2 = \frac{1}{n}\sum_{ij}a_{ij}^2$
    .
    One can take $C_1=C_2=36$
\end{theorem}
\begin{remark}
    As in Bercu--Deylon--Rio~\cite{MR3363542}, note that setting
    \[
        d_{ij} =
        a_{ij}
        -
        \frac{1}{n}
        \sum_{k=1}^n\bigl(
            a_{ik} + a_{kj}
        \bigr)
        +
        \frac{1}{n^2}
        \sum_{k,l=1}^n a_{kl}
    \]
    yields 
    $
        \Var(f) =
        \frac{1}{n-1}
        \sum_{ij} d_{ij}^2
    $
    and 
    $
        f-\EE f = \sum_{k=1}^n d_{k\sigma(k)}.
    $
    Therefore, an application of Theorem~\ref{T:f-Bennett-general} to $(f-\EE f)/2$ in place of $f$ (note that $\sum_{ij}d_{ij}=0$, while $a_{ij}\in[-1,1]$ are arbitrary) provides that 
    \begin{equation}\label{eq:f-Bennett-general-centered}
        \forall\; t\ge 0
        \qquad 
        \PP(f\ge \EE f + t)
        \le 
        2\exp\Bigl(
            -\frac{t}{2C_1}
            \log\Bigl(
                1 + \frac{t}{2C_2\Var(f)}
            \Bigr)
        \Bigr).
    \end{equation}
    As shown by Hoeffding in~\cite{MR44058} (cf. also Bolthausen~\cite{MR751577} for a Stein method based approach), as soon as 
    \[
      \lim_{n\to\infty}
      \frac{\max_{i,j\in[n]} d_{ij}}{\Var(S_n)} = 0,  
    \]
    then $f$ verifies the CLT, i.e.,
    \[
      \frac{f-\EE f}{\sqrt{\Var(f)}}
        \overset{n\to\infty}\longrightarrow  
        \mathcal{N}(0,1)
    \]
    in law.
    Clearly, the bound from~\eqref{eq:f-Bennett-general-centered} becomes subgaussian  for small values of $t$ and whence matches the CLT behavior described above (up to numerical constants).
    Similarly, if one chooses $a_{ij}=\indbr{i=j}$, then $f$ becomes the number of fixed  points of a random permutation $\sigma$.
    The exact tail distribution of $f$ in such case is well known, cf.~\cite[Section IV.4]{feller2008introduction}, and is of order $\exp(-Ct\log t)$ for $t$ big and some $C>0$, which agrees with the bound~\eqref{eq:f-Bennett-general-centered}.
    This shows that Theorem~\ref{T:f-Bennett-general} is optimal up to the numerical  constants.
\end{remark}
To prove the Bennett inequality of Theorem~\ref{T:f-Bennett-general}, we first derive it for non-negative statistics in the theorem below. 
\begin{theorem}\label{T:f-Bennett-positive}
    Let $f$ be given by~\eqref{eq:Hoeffding-statistic}.
    If $a_{ij}\in [0,1]$ for all $i,j$, then 
    \begin{equation*}
        \forall\; t\ge 0 
        \qquad 
        \PP(f> \EE f + t)
        \le 
        \exp\Bigl(
            -\frac{t}{4}
            \log\Bigl(
                1+
                \frac{t}{4\EE f}
            \Bigr)
        \Bigr).
    \end{equation*}
\end{theorem}
\begin{remark}
    Theorem~\ref{T:f-Bennett-positive} already improves (up to numerical constants in the exponent) upon a Bernstein-type bound
    \[
        \forall\; t\ge 0 
        \qquad 
        \PP(f> \EE f + t)
        \le 
        \exp\Bigl(
            -
            \frac{t^2}{4\EE f + 2t}
        \Bigr)
    \] 
    obtained by Chatterjee~\cite[Proposition 1.1]{MR2288072}.
\end{remark}
\begin{proof}[Proof of Theorem~\ref{T:f-Bennett-positive}]
    Since $a_{ij}\in[0,1]$, then for any $i,j$,
    \begin{equation}\label{eq:f-fij_sum}
        \sum_{ij}
        (f_{ij}-f)_+
        =
        \sum_{ij}
        (a_{i\sigma_j}+a_{j\sigma_i} - a_{i\sigma_i}-a_{j\sigma_j})_+
        \le 
        \sum_{ij}
        (a_{i\sigma_j}+a_{j\sigma_i})
        =
        2\sum_{ij}a_{ij}
        =
        2n\EE f.
    \end{equation}
    By the modified log-Sobolev inequality, using~\eqref{eq:f-fij_sum} and convexity of $x\mapsto e^{2x}$, we arrive at
    \begin{align*}
        \Ent(e^{\lambda f})
        &\le 
        \frac{\lambda}{n}
        \EE e^{\lambda f}
        \sum_{ij}
        (e^{\lambda (f_{ij}-f)_+}-1)
        (f_{ij}-f)_+
        \\&\le 
        \frac{\lambda}{n}
        (e^{2\lambda}-1)
        \EE e^{\lambda f}
        \sum_{ij}(f_{ij}-f)_+
        \\&\le 
        2\lambda(e^{2\lambda}-1)
        \EE f
        \EE e^{\lambda f}
        \\&\le 
        4\lambda^2 e^{2\lambda}
        \EE f
        \EE e^{\lambda f}
    \end{align*}
    for all $\lambda\ge 0$.
    Hence, using Proposition~\ref{P:Herbst_Poisson} with $a=4\EE f$, $b=2$ gives the conclusion.
\end{proof}

Finally, to prove Theorem~\ref{T:f-Bennett-general}, we need the following proposition.
We defer its proof to Section~\ref{sec:missing-proofs}.
\begin{proposition}\label{L:f-Bernstein-general}
    Let $f$ be given by~\eqref{eq:Hoeffding-statistic}.
    If $a_{ij} \in [-1,1]$ for all $i,j$, then 
    \[
        \forall\; t\ge 0
        \qquad 
        \PP(f\ge \EE f + t)
        \le 
        \exp\Bigl(
            -\min\Bigl(
                \frac{t}{32},
                \frac{t^2}{128\EE{\Sigma}^2}
            \Bigr)
        \Bigr),
    \]
    where $\Sigma^2=\sum_{k}a_{k\sigma_k}^2$ so that 
    $\EE \Sigma^2 = \frac{1}{n}\sum_{ij}a_{ij}^2$.
\end{proposition}
\begin{proof}[Proof of Theorem~\ref{T:f-Bennett-general}]
    For a fixed $t > 0$, set
    \[
        \rho^{-1}
        =
        2
        \log\bigl(
            1+\frac{t}{16 \EE \Sigma^2}
        \bigr)
    \] 
    and denote 
    \[
        f^\downarrow(\sigma)
        = 
        \sum_{i}
        a_{i\sigma_i}
        \indbr{\abs{a_{i\sigma_i}}\le \rho}
    \]  
    and 
    \[
        f^\uparrow(\sigma)
        =
        \sum_i
        \abs{a_{i\sigma_i}}
        \indbr{\abs{a_{i\sigma_i}} > \rho}
    \]
    so that $f\le f^\downarrow + f^\uparrow$.
    We estimate the tail probabilities for $f^\downarrow$ and $f^\uparrow$.

    By the estimate $\log(1+x)\le x$ and by the definition of $\rho$,
    \begin{equation*}
        \frac{t}{32\rho}
        \le 
        \frac{t^2}{256\EE \Sigma^2}
        \le 
        \frac{t^2}{128\EE \Sigma^2},
    \end{equation*}
    whence by Proposition~\ref{L:f-Bernstein-general} applied to $f^\downarrow/\rho$,
    \begin{align}\label{eq:f-downarrow-tail}
        \begin{split}
        \PP(f^\downarrow \ge \EE f^\downarrow + t)
        &\le 
        \exp\Bigl(
            -\min\Bigl(
                \frac{t}{32\rho},
                \frac{t^2}{128\EE \Sigma^2}
            \Bigr)
        \Bigr)
        \\&=
        \exp\Bigl(
            -\frac{t}{32\rho}
        \Bigr)
        =
        \exp\Bigl(
            -
            \frac{t}{16}
            \log\bigl(
                1+\frac{t}{16\EE \Sigma^2}
            \bigr)
        \Bigr).
        \end{split}
    \end{align}
    By the definitions of $f^\uparrow, \rho$ and estimate $\log(1+x)\le 2\log(1+\sqrt{x})\le 2\sqrt{x}$,
    \[
        \EE f^\uparrow 
        \le 
        \frac{\EE \Sigma^2}{\rho}  
        =
        2(\EE \Sigma^2)
        \log\Bigl(
            1+\frac{t}{16\EE\Sigma^2}
        \Bigr)
        \le 
        \sqrt{t\EE\Sigma^2},
    \]
    whence by Theorem~\ref{T:f-Bennett-positive} applied to $f^\uparrow$,
    \begin{align}\label{eq:f-uparrow-tail}
        \begin{split}
        \PP(f^\uparrow \ge \EE f^\uparrow + t)
        &\le  
        \exp\Bigl(
            -\frac{t}{4}
            \log\Bigl(
                1+\frac{t}{4\EE f^\uparrow}
            \Bigr)
        \Bigr)
        \\&\le  
        \exp\Bigl(
            -\frac{t}{4}
            \log\Bigl(
                1+
                \frac{1}{4}
                \sqrt{
                    \frac{t}{\EE \Sigma^2}
                }
            \Bigr)
        \Bigr)
        \\&\le  
        \exp\Bigl(
            -\frac{t}{8}
            \log\Bigl(
                1+
                \frac{t}{16\EE \Sigma^2}
            \Bigr)
        \Bigr),
        \end{split}
    \end{align}
    where in the last step we have used again the estimate $2\log(1+\sqrt{x})\ge \log(1+x)$.
    Using the assumption $\EE f =0$, triangle inequality and estimating $\log(1+x)\le x$, we obtain 
    \begin{equation}\label{eq:f-arrow-mean}
        \abs{\EE f^\downarrow}
        =
        \abs{\EE f^\downarrow - \EE f }
        \le 
        \EE f^\uparrow
        \le 
        \frac{\EE \Sigma^2}{\rho}
        \le 
        \frac{1}{8}t.
    \end{equation}
    By combining~\eqref{eq:f-downarrow-tail},~\eqref{eq:f-uparrow-tail} and~\eqref{eq:f-arrow-mean} we arrive at
    \begin{align*}
        \PP(f\ge 9t/4)
        &\le 
        \PP(f^\downarrow \ge 9t/8)
        +
        \PP(f^\uparrow \ge 9t/8)
        \\&\le 
        \PP(f^\downarrow \ge \EE f^\downarrow + t)
        +
        \PP(f^\uparrow \ge \EE f^\uparrow + t)
        \\&\le 
        2
        \exp\Bigl(
            -\frac{t}{16}
            \log\Bigl(
                1+
                \frac{t}{16\EE \Sigma^2}
            \Bigr)
        \Bigr).
    \end{align*}
    Substituting $t \leftarrow 4t/9$ yields the result.
\end{proof}

\section{Proof of Propositions~\ref{L:Z-Bernstein-general} and~\ref{L:f-Bernstein-general}}\label{sec:missing-proofs}
Both propositions are special cases of a more general result for suprema of Hoeffding statistics which we provide below.
Let $R\subset\RR^{n\times n}$ be a set of real matrices.
Denote
\begin{equation}\label{eq:S_def}
  S = \sup_{r\in R} \sum_{k=1}^n r_{k\sigma_k}.
\end{equation}
The main result of this section is the following estimate.
\begin{proposition}\label{P:S-Bernstein-general}
    Let $S$ be given by~\eqref{eq:S_def} and assume
    $R \subset [-1,1]^{n\times n}$.
    Then
    \[
        \forall\; t\ge 0
        \qquad 
        \PP(S \ge \EE S + t)
        \le 
        \exp\Bigl(
            -\min\Bigl(
                \frac{t}{32},
                \frac{t^2}{128 \EE 
                \Sigma_R^2}
            \Bigr)
        \Bigr),
    \]
    where 
    $\Sigma_R^2 = \sup_{r\in R} \sum_k r_{k\sigma_k}^2$.
\end{proposition}
Propositions~\ref{L:Z-Bernstein-general} and~\ref{L:f-Bernstein-general} are special cases of Proposition~\ref{P:S-Bernstein-general} as illustrated below.
\begin{proof}[Proof of Proposition~\ref{L:Z-Bernstein-general}]
    Apply Proposition~\ref{P:S-Bernstein-general} with $R = \{\, a^x \colon x\in \calX \,\}$ (recall the definition of the matrix $a^x$ introduced at the beginning of Section~\ref{sec:suprema}).
    \end{proof}
    \begin{proof}[Proof of Propositoin~\ref{L:f-Bernstein-general}]
    Apply Proposition~\ref{P:S-Bernstein-general} with $R = \{ a \}$.
    \end{proof}
To prove Proposition~\ref{P:S-Bernstein-general}, let us first state the modified log-Sobolev inequality~\eqref{eq:mLS} for the Laplace transform of $S$.
For any $i,j\in[n]$, denote 
\begin{equation*}
    S_{ij} = \sup_{r\in R} \sum_{k=1}^n r_{k\sigma_{ij}(k)}.
\end{equation*} 
Then, the modified log-Sobolev inequality~\eqref{eq:mLS} implies that
\begin{equation*}
    \Ent(e^{\lambda S})
    \le 
    \frac{\lambda}{n}
    \EE 
    \Bigl[
        e^{\lambda S}
        \sum_{ij}(1-e^{-\lambda(S-S_{ij})})_+(S-S_{ij})_+
    \Bigr],
\end{equation*}
which after estimating $1-e^{-x}\le x$ can be further specialized to 
\begin{equation}\label{eq:mls-S}
    \Ent(e^{\lambda S})
    \le 
    \frac{\lambda}{n}
    \EE 
    \Bigl[
        e^{\lambda S}
        \sum_{ij}(S-S_{ij})_+^2
    \Bigr].
\end{equation}
We need also the following auxiliary fact.

\begin{lemma}\label{L:S-Bernstein-positive}
    Let $S$ be given by~\eqref{eq:S_def} and assume $R \subset[0,1]^{n\times n}$.
    Then 
    \[
        \forall\;\lambda\in [0,1/4]
        \qquad
        \log \EE e^{\lambda S} \le 
        2\lambda\EE S
        .
    \]
\end{lemma}
\begin{proof}
    Assume w.l.o.g. that $R$ is finite.
    Let $\hat{r}$ be a random matrix taking values in $R$ such that $S=\sum_{k=1}^n \hat{r}_{k\sigma_k}$. 
    We have
    \begin{align}
        \begin{split}\label{eq:S-Sij2}
        \sum_{ij} (S-S_{ij})_+^2
        &\le
        \sum_{ij} (
            \hat{r}_{i\sigma_i} 
            + \hat{r}_{j\sigma_j} 
            - \hat{r}_{i\sigma_j} 
            - \hat{r}_{j\sigma_i})_+^2
        \\&\le 
        \sum_{ij} (\hat{r}_{i\sigma_i} + \hat{r}_{j\sigma_j})^2
        \le 
        2n \sum_i (\hat{r}_{i\sigma_i})^2
        \le 
        2nS,
        \end{split}
    \end{align}
    where in the last inequality we have used that $R\in[0,1]^{n\times n}$.

    By the modified log-Sobolev inequality~\eqref{eq:mls-S} combined with~\eqref{eq:S-Sij2}, we arrive at
    \begin{align*}
        \Ent(e^{\lambda S})
        &\le 
        \frac{\lambda^2}{n}
        \EE 
        \Bigl[
            e^{\lambda S}
            \sum_{ij}
            (S-S_{ij})_+^2
        \Bigr]
        \le 
        2\lambda^2
        \EE [e^{\lambda S}S]
    \end{align*}
    for all $\lambda\ge 0$.
    Applying Proposition~\ref{P:Herbst_Bernstein} with $a=2$, $b=0$ results in 
    \[
        (
            1-2\lambda
        )
        \log \EE e^{\lambda S}
        \le 
        \lambda \EE S,
    \] 
    for all $\lambda\ge 0$, which yields the conclusion.
\end{proof}
We are in position to prove Proposition~\ref{P:S-Bernstein-general}.
\begin{proof}[Proof of Proposition~\ref{P:S-Bernstein-general}]
    Let $\hat{r}$ be a random matrix taking values in $R$ such that $S=\sum_{k=1}^n \hat{r}_{k\sigma_k}$.
    By the triangle inequality in $\ell^2$,
    \begin{align}\label{eq:S-Sij2-sigma-pre}
        \begin{split}
        \sum_{ij}(S-S_{ij})_+^2
        &\le
        \sum_{ij}
        (
            \hat{r}_{i\sigma_i}
            + \hat{r}_{j\sigma_j} 
            - \hat{r}_{i\sigma_j}
            - \hat{r}_{j\sigma_i}
        )_+^2
        \\&\le 
        8\sum_{ij} \hat{r}_{i\sigma_i}^2
        +
        8\sum_{ij} \hat{r}_{i\sigma_j}^2
        \le 
        8n\Sigma_R^2
        + 
        8\sum_{ij} \hat{r}_{i\sigma_j}^2.
        \end{split}
    \end{align}
    Note that
    \[
        \sum_{ij} \hat{r}_{i\sigma_j}^2
        =
        \sum_{ij} \hat{r}_{ij}^2
        =
        n
        \EE 
        \sum_{i}
        \hat{r}_{i\sigma_i}^2
        \le 
        n
        \EE 
        \sup_{r\in R}
        \sum_{i}
        r_{i\sigma_i}^2
        =
        n\EE\Sigma_R^2,
    \]
    whence~\eqref{eq:S-Sij2-sigma-pre} can be further specialized to 
    \begin{equation}\label{eq:S-Sij2-sigma}
        \sum_{ij}(S-S_{ij})_+^2
        \le
        8n(\Sigma_R^2 + \EE \Sigma_R^2).
    \end{equation}
    By the modified log-Sobolev inequality~\eqref{eq:mls-S} combined with~\eqref{eq:S-Sij2-sigma}, we arrive at
    \begin{align}\label{eq:Z-Bernstein-general-mls}
        \begin{split}
        \Ent(e^{\lambda S})
        &\le 
        \frac{\lambda^2}{n}
        \EE \Bigl[ 
            e^{\lambda S} 
            \sum_{ij}
            (S-S_{ij})_+^2
        \Bigr]
        \le 
        8\lambda^2
        \bigl(
            (\EE e^{\lambda S})(\EE  {\Sigma}_R^2)
            +
            \EE [e^{\lambda S}\Sigma_R^2]
        \bigr)
        .
        \end{split}
    \end{align}
    Recall the variational formula for entropy $\Ent(h)
    =
    \sup\bigl\{\,
        \EE hg\colon
        \EE e^g\le 1  
    \,
    \bigr\}$, from which it follows that for any $h,g$
    \begin{equation}\label{eq:entropy_variational2}
        \EE hg 
        \le
        \Ent(h)
        +
        (\EE h)\log(\EE e^g)
        .
    \end{equation}
    Applying first~\eqref{eq:entropy_variational2} with $h=e^{\lambda S}$, $g=\Sigma_R^2/4$ and then Lemma~\ref{L:S-Bernstein-positive} yields
    \begin{equation*}
        \EE \bigl[e^{\lambda S}\Sigma_R^2\bigr]
        \le 
        4\Ent(e^{\lambda S})
        +
        4
        \bigl(
            \EE e^{\lambda S}
        \bigr)
        \bigl(
            \log \EE e^{\Sigma_R^2/4}
        \bigr)
        \le 
        4\Ent(e^{\lambda S})
        +
        2
        \bigl(
            \EE e^{\lambda S}
        \bigr)
        \bigl(
            \EE\Sigma_R^2
        \bigr)
        ,
    \end{equation*}
    which combined with~\eqref{eq:Z-Bernstein-general-mls} results in
    \begin{align*}
        (1-32\lambda^2)\Ent(e^{\lambda S})
        \le 
        24\lambda^2
        \bigl( 
            \EE {\Sigma_R}^2
        \bigr)
        \bigl(
            \EE e^{\lambda S}
        \bigr)
    \end{align*}
    for all $\lambda\ge 0$, so that 
    \begin{align*}
        \Ent(e^{\lambda S})
        \le 
        \frac{192}{7}\lambda^2
        \bigl( 
            \EE {\Sigma_R}^2
        \bigr)
        \bigl(
            \EE e^{\lambda S}
        \bigr)
        \le 
        32\lambda^2
        \bigl( 
            \EE {\Sigma_R}^2
        \bigr)
        \bigl(
            \EE e^{\lambda S}
        \bigr)
    \end{align*}
    for all $\lambda\in[0, 1/16]$.
    We conclude by applying Proposition~\ref{P:Herbst_Bernstein2} with $\varepsilon=\frac{1}{16}$ and $b=32\EE {\Sigma}^2_R$.
\end{proof}
\section{Acknowledgements}
I would like to thank Rados{\l}aw Adamczak for reading thoroughly the initial versions of this manuscript and for his numerous suggestions which significantly improved its quality.

\bibliographystyle{amsplain}
\bibliography{sampling}
\newpage
\appendix

\section{Proof of Lemma~\ref{L:hoeffding}}\label{sec:hoeffding}
Set $E=\RR^n$ and $g(i)=e_i$, where $e_i\in\RR^n$ is a vector with 1 on the
$i$-th coordinate and 0's elsewhere.
Moreover, let for any $v\in\RR^n$
\[
    \Psi(v) = 
    \phi\Bigl( \,
        \sup_{x\in\calX} \langle x, v \rangle
    \Bigr),
\]
where $\langle\cdot,\cdot\rangle$ is the standard dot product.
Then, 
\[  
    \phi(Z) 
    =
    \phi\Bigl(
        \,\sup_{x\in\calX} 
        \langle x, \sum_{k=1}^m e_{I_k}\rangle
    \Bigr)
    = 
    \Psi\Bigl(\sum_{k=1}^m g(I_k)\Bigr)
\]  
and identically $\phi(Z') = \Psi(\sum_{k=1}^m g(J_k))$.
Finally, for any $v,w\in\RR^n$ and $t\in[0,1]$
\begin{align*}
    \Psi(tw+(1-t)v)
    &= 
    \phi\Bigl(\,
        \sup_{x\in\calX}
        \langle x, tw+(1-t)v \rangle
    \Bigr)
    \\&\le 
    \phi\Bigl(\,
        t\sup_{x\in\calX}
        \langle x, w \rangle
        +
        (1-t)\sup_{x\in\calX}
        \langle x, v \rangle
    \Bigr)
    \le  
    t\Psi(w)
    +
    (1-t)\Psi(v),
\end{align*}
where in the first inequality we have used that $\phi$ is increasing, and in the second inequality we have used that $\phi$ is convex.
We conclude by applying Hoeffding's argument~\eqref{eq:Hoeffding_argument} to the pair $(g,\Psi)$.

\section{Proof of Lemma~\ref{T:Z-Bennett-positive}}\label{sec:Ledoux-pf}
Let us recall some facts regarding entropy.
For any random variable $Y$ measurable w.r.t. $\sigma(J_1,\ldots,J_m)$ and any $k\in[m]$, 
let $\EE^{(k)}$ denote the expectation w.r.t. $J_k$ only, i.e.,
\[
    \EE^{(k)}[Y] = 
    \EE \bigl[
        Y\,\vert\, J_1,\ldots,J_{k-1},J_{k+1},\ldots,J_m
    \bigr].
\]
For such positive $Y$, recall the tensorization of entropy formula (cf., e.g.,~\cite[Theorem 4.10]{MR3185193})
\begin{equation}\label{eq:tensorization}
    \Ent(Y)
    \le
    \EE 
    \sum_{k=1}^m 
    \Ent^{(k)}(Y),
\end{equation}
where 
\[
    \Ent^{(k)}(Y) 
    =
    \EE^{(k)} \bigl[
        Y\log Y
    \bigr]  
    -
    \EE^{(k)} \bigl[
        Y
    \bigr]  
    \log 
    \EE^{(k)} \bigl[
        Y
    \bigr] 
\]
is the entropy functional corresponding to $\EE^{(k)}$.
Moreover, recall the following variational formula for the entropy
\begin{equation}\label{eq:variational}
    \Ent(Y) = 
    \inf_{c>0}
    \EE\Bigl[
        Y(\log Y - \log c)
        -
        (Y-c)
    \Bigr].
\end{equation}
\begin{proof}[Proof of Lemma~\ref{T:Z-Bennett-positive}]
    For $k\in[m]$, let
    \[
        Z_k' 
        = 
        \sup_{x\in\calX} \sum_{l=1,l\neq k}^m   x_{J_l}
    \]
    (if $m=1$, then we put $u_1=0$). 
    By the tensorization of entropy~\eqref{eq:tensorization} and by~\eqref{eq:variational},
    \begin{align}\label{eq:entZ'-initial-est}
    \begin{split}
        \Ent(e^{\lambda Z'})
        &\le 
        \EE 
        \sum_{k=1}^m 
        \Ent^{(k)}(e^{\lambda Z'})
        \\&=
        \EE 
        \sum_{k=1}^m 
        \inf_{c_k>0}
        \EE^{(k)}\Bigl[
            e^{\lambda Z'}(\lambda Z' - \log c_k)
            -
            (e^{\lambda Z'} - c_k)
        \Bigr]
        \\&\le 
        \EE 
        \sum_{k=1}^m 
        \EE^{(k)}\Bigl[
            e^{\lambda Z'}( \lambda Z' - \lambda Z'_k)
            -
            (e^{\lambda Z'} - e^{\lambda Z'_k})
        \Bigr]
        \\&\le 
        \EE \Bigl[
            e^{\lambda Z'}
            \sum_{k=1}^m 
            \phi( -\lambda (Z' -Z'_k) )
        \Bigr],
    \end{split}
    \end{align}
    where $\phi(z)=e^{z}-z-1$.

    Note that 
    \begin{equation*}
        \sum_{k=1}^m 
        (Z' - Z_k')
        \le 
        Z' 
    \end{equation*}
    and that for any $z\in[0,1]$ and $\lambda \ge 1/4$, by the convexity of the function $z\mapsto e^{-z/4}-1$
    \begin{equation*}
        \phi(-\lambda z)
        =
        e^{-\lambda z}-1+\lambda z
        \le 
        e^{-z/4}-1+\lambda z
        \le 
        -\frac{z}{4}e^{-1/4}+\lambda z
        \le 
        \Bigl(
            \lambda-\frac{1}{8}
        \Bigr)z.
    \end{equation*}
    Since $\calX\subset [0,1]^n$ by assumption, therefore $0\le Z' - Z_k' \le 1$ and whence we can estimate~\eqref{eq:entZ'-initial-est} further for any $\lambda\ge 1/4$ as follows,
    \begin{align*} 
    \begin{split}
        \Ent(e^{\lambda Z'})
        \le 
        \Bigl(
            \lambda - \frac{1}{8}
        \Bigr)
        \EE \Bigl[
        e^{\lambda Z'}
        \sum_{k=1}^m 
        ( Z' - Z'_k)
        \Bigr]
        \le
        \Bigl(
            \lambda - \frac{1}{8}
        \Bigr)
        \EE
        \bigl[
            e^{\lambda Z'}Z'
        \bigr],
    \end{split}
    \end{align*}
    which after rearrangement yields
    \[
      \EE \bigl[e^{\lambda Z'}Z'\bigr]
      \le 
      8
      \EE e^{\lambda Z'}
      \log \EE e^{\lambda Z'},
    \]
    which in turn is equivalent to
    \[
      \frac{d}{d\lambda}
      \bigl(
        \log \EE e^{\lambda Z'}
      \bigr)  
      \le 
      8
      \log \EE e^{\lambda Z'}
    \]
    for any $\lambda \ge 1/4$.
    Integrating w.r.t. $\lambda$ yields that 
    \begin{equation}\label{eq:ledoux_pre_final_est}
        \log \EE e^{\lambda Z'} 
        \le 
        e^{8\lambda-2}
        \log \EE e^{Z'/4}.
    \end{equation}
    
    We turn to estimating the term $\log \EE e^{Z'/4}$.
    Using again that $0\le Z' - Z_k' \le 1$, we obtain that
    \begin{equation*}
        \sum_{k=1}^m 
        (Z' - Z_k')^2
        \le 
        Z'.
    \end{equation*}
    Moreover, by comparing the derivatives, we get that for any $z\ge 0$,
    \[
        \phi(-z) \le \frac{z^2}{2}
    \]
    and thus we can also estimate further~\eqref{eq:entZ'-initial-est} as 
    \begin{equation*}
        \Ent(e^{\lambda Z'})
        \le 
        \frac{\lambda^2}{2}
        \EE \Bigl[
            e^{\lambda Z'}
            \sum_{k=1}^m 
            (Z' -Z'_k)^2
        \Bigr]
        \le 
        \frac{\lambda^2}{2}
        \EE \bigl[
            e^{\lambda Z'}
            Z'
        \bigr].
    \end{equation*}
    Applying Proposition~\ref{P:Herbst_Bernstein} with $a=\frac{1}{2}$ and $b=0$ yields that 
    \[
        \forall\;\lambda\ge 0\qquad 
        \Bigl(1-\frac{\lambda}{2}\Bigr)
        \log \EE e^{\lambda Z'}
        \le 
        \lambda \EE Z'
    \]
    so that 
    \[
        \forall\; 
        \lambda\in[0,1/4]
        \qquad 
        \log\EE e^{\lambda Z'}
        \le 
        \frac{8}{7}
        \lambda
        \EE Z',
    \]
    which combined with~\eqref{eq:ledoux_pre_final_est} yields
    \[
        \log \EE e^{\lambda Z'} 
        \le 
        \frac{2}{7e^2}
        e^{8\lambda}
        \EE Z'\le 
        \frac{1}{16}
        e^{8\lambda}
        \EE Z'
    \]
    as desired.
\end{proof}

\section{Variants of the Herbst argument}
Throughout this section, $X$ is a random variable such that its Laplace transform $F$ is well defined on $[0,\infty)$. 
In that case, recall that 
\[
    \Ent(e^{\lambda X})
    =
    \lambda F'(\lambda)
    -
    F(\lambda)\log F(\lambda)
\]
for all $\lambda\ge 0$.
Below we gather some variants of the celebrated Herbst argument.
\begin{proposition}\label{P:Herbst_Poisson}
    If for any $\lambda\ge 0$,
    \begin{equation}\label{eq:Herbst_Poisson_asm}
        \lambda F'(\lambda) - F(\lambda)\log F(\lambda)
        \le 
        a\lambda^2e^{b\lambda}F(\lambda)
    \end{equation}
    for some $a,b>0$, then 
    \begin{equation}\label{eq:Herbst_Poisson_Laplace}
        \forall\;\lambda \ge 0
        \qquad
        \log\EE e^{\lambda(X-\EE X)}
        \le 
        \frac{a}{b}\lambda(e^{b\lambda}-1)
    \end{equation}
    and in particular
    \begin{equation}\label{eq:Herbst_Poisson}
        \forall\; t\ge 0 
        \qquad 
        \PP\bigl(
            X \ge \EE X + t
        \bigr)
        \le 
        \exp\Bigl(
            -\frac{t}{2b}
            \log\Bigl(
                1+\frac{b}{2a}t
            \Bigr)
        \Bigr).
    \end{equation}
\end{proposition}
\begin{proof}
    Set $H(\lambda) = \frac{\log F(\lambda)}{\lambda}$ for $\lambda > 0$.
    Then,~\eqref{eq:Herbst_Poisson_asm} implies 
    $H'(\lambda) \le ae^{b\lambda}$.
    Since $H(0^+) = \EE X$, then for any $\lambda > 0$,
    \[
        H(\lambda) \le \EE X + \frac{a}{b}(e^{b\lambda}-1),         
    \]
    which translates to~\eqref{eq:Herbst_Poisson_Laplace} and consequently, by the Chernoff bound
    \[
        \PP\bigl(
            X \ge \EE X + t
        \bigr)
        \le 
        \inf_{\lambda > 0}
        \exp\Bigl(
            -\lambda t
            +
            \frac{a}{b}\lambda
            (e^{b\lambda}-1)
        \Bigr)
    \]
    for all $t\ge 0$.
    Choosing $\lambda = \frac{1}{b}\log(1+\frac{b}{2a}t)$ yields~\eqref{eq:Herbst_Poisson}.
\end{proof}

\begin{proposition}\label{P:Herbst_Bernstein}
    Assume that for all $\lambda\ge 0$,
    \begin{equation}\label{eq:Herbst_Bernstein_asm}
        \lambda F'(\lambda) - F(\lambda)\log F(\lambda)
        \le 
        \lambda^2\bigl(
            aF'(\lambda)
            +
            b F(\lambda)
        \bigr)
    \end{equation}
    for some $a,b\in\RR$. 
    Then 
    \begin{equation}\label{eq:Herbst_Bernstein_Laplace}
        \forall\; \lambda \ge 0
        \qquad
        (1-a\lambda)
        \log \EE e^{\lambda X}
        \le 
        \lambda \EE X
        +
        b\lambda^2.
    \end{equation}
    If additionally $a>0$ and $X$ is not constant, then $a\EE X + b > 0$ and
    \begin{equation}\label{eq:Herbst_Bernstein}
        \forall\; t\ge 0 
        \qquad 
        \PP\bigl(
            X \ge \EE X + t
        \bigr)
        \le 
        \exp\Bigl(
            -
            \min\Bigl(
            \frac{t}{4a},
            \frac{t^2}{8(a\EE X + b)}   
            \Bigr)
        \Bigr).
    \end{equation}
\end{proposition}
\begin{proof}
    Set $H(\lambda) = \frac{\log F(\lambda)}{\lambda}$ for $\lambda > 0$.
    Then,~\eqref{eq:Herbst_Bernstein_asm} implies 
    \[
        H'(\lambda) \le 
        a \frac{F'(\lambda)}{F(\lambda)}
        +
        b
        =
        \frac{d}{d\lambda}
        \bigl(
            a\log F(\lambda)
            +
            b\lambda
        \bigr).
    \]
    Consequently, for any $\lambda > 0$,
    \[
        H(\lambda) \le 
        H(0^+)
        +
        a \log F(\lambda)
        +
        b\lambda,
    \]
    which is equivalent to~\eqref{eq:Herbst_Bernstein_Laplace} since $H(0^+) = \EE X$.
    Subtracting $(1-a\lambda)\lambda\EE X$ from both sides gives
    \begin{equation}\label{eq:1-a_cond}
        (1-a\lambda)
        \log \EE e^{\lambda(X-\EE X)}
        \le 
        \lambda^2(a\EE X +b).
    \end{equation}
    By Jensen's inequality and the fact that $X$ is not constant, $\log \EE e^{\lambda(X-\EE X)} > 0$.
    If $\lambda \le 1/2a$, then $1/2\le 1-a\lambda$, whence~\eqref{eq:1-a_cond} implies
    \[
        \forall\;\lambda \in[0,1/2a]
        \qquad 
        0 <
        \log \EE e^{\lambda(X-\EE X)}
        \le 
        2
        \lambda^2(a\EE X +b).
    \]
    Therefore, by the Chernoff bound
    \[
        \PP\bigl(
            X \ge \EE X + t
        \bigr)
        \le 
        \inf_{0 \le \lambda \le 1/2a}
        \exp\Bigl(
            -\lambda t
            +
            2\lambda^2(a\EE X+b)
        \Bigr)
    \]
    for all $t\ge 0$.
    Choosing $\lambda = \frac{t}{4(a\EE X + b)}$ if $t\le \frac{2(a\EE X + b)}{a}$ and $\lambda = \frac{1}{2a}$ otherwise yields~\eqref{eq:Herbst_Bernstein}.
\end{proof}

\begin{proposition}\label{P:Herbst_Bernstein2}
    Assume that for some $\varepsilon,b>0$ and all $\lambda\in[0,\varepsilon]$,
    \begin{equation}\label{eq:Herbst_Bernstein2_asm}
        \lambda F'(\lambda)
        -
        F(\lambda)\log F(\lambda)
        \le 
        b
        \lambda^2 F(\lambda).
    \end{equation}
    Then 
    \begin{equation}\label{eq:Herbst_Bernstein2}
        \forall\; t\ge 0 
        \qquad 
        \PP\bigl(
            X \ge \EE X + t
        \bigr)
        \le 
        \exp\Bigl(
            -
            \min\Bigl(
            \frac{\varepsilon t}{2},
            \frac{t^2}{4b}   
            \Bigr)
        \Bigr).
    \end{equation}
\end{proposition}
\begin{proof}
    Dividing~\eqref{eq:Herbst_Bernstein2_asm} by $\lambda^2 F(\lambda)$ and integrating w.r.t. $\lambda$ yields 
    \[
      \frac{\log \EE e^{\lambda X}}{\lambda}
      \le 
      \EE X
      +
      \lambda b  
    \]
    for all $\lambda\in [0,\varepsilon]$.
    Therefore, by the Chernoff bound
    \[
        \PP\bigl(
            X \ge \EE X + t
        \bigr)
        \le 
        \inf_{0 \le \lambda \le \varepsilon}
        \exp\bigl(
            -\lambda t
            +
            b\lambda^2
        \bigr)
    \]
    for all $t\ge 0$.
    Choosing $\lambda = \frac{t}{2b}$ if $t\le 2b\varepsilon$ and $\lambda = \varepsilon$ otherwise yields~\eqref{eq:Herbst_Bernstein2}.
\end{proof}

\section{Example}\label{sec:example}

In this section we provide an example showing how our result of Theorem~\ref{T:Z-Bennett-general} can improve upon the bound by Tolstikhin--Blanchard--Kloft~\cite{tolstikhin2014localized} in the case of non-symmetric set $\calX$, cf. Remark~\ref{R:TBK-comparision}.

For some $k,l\in\NN$ (to be determined lated) such that $0<l\le k\le n/2$, let $A,B\subset [n]$ be two \emph{disjoint} sets of cardinalities $k$ and $k/2$ respectively and set 
\[
  \calX = \{\,
    {\bf 1}_{S} - {\bf 1}_{B}
    \;\colon\;
    S\subset A,\enspace
    \vert S \vert \le l
  \,\}.
\]
For any set $S\subset [n]$, denote 
\[
  R_S
  = 
  \vert\{\,
    j\in[m]\;\colon\; I_j\in S
  \,\}\vert,
  \quad 
  \widetilde{R}_S
  = 
  \vert\{\,
    j\in[m]\;\colon\; J_j\in S
  \,\}\vert
\]
so that $Z=\min(R_A,l)-R_B$ and $\Sigma^2 = \min(R_A,l) + R_B$.
Note that 
$
    \EE R_S = 
    \EE \widetilde{R}_S =
    \frac{m\vert S \vert}{n}
$
for any set $S\subset [n]$
and thus 
\begin{align*}
  \frac{mk}{n}
  = 
  \EE R_B
  \le 
  \EE \Sigma^2 
  \le 
  \EE \widetilde{\Sigma}^2
  \le 
  \EE R_A + \EE R_B
  =
  \frac{3mk}{2n}.
\end{align*}

Let moreover
\[
  W = \vert\{\,
    i\in A \;\colon\; \exists\;j\in[m]\enspace J_j=i
  \,\}\vert
\]
denote the number of elements sampled from the set $A$ in the sampling with replacement scheme.
Then, on the set $\{ W \le l \}$, $Z'=\widetilde{R}_A-\widetilde{R}_B$ and $\widetilde{\Sigma}^2=\widetilde{R}_A+\widetilde{R}_B$.
Choose any $m\simeq\frac{n}{2}$, where we use the notation $x_n\simeq y_n$ if $x_n=y_n(1+o(1))$.
We first show that $\{ W \le l \}$ occurs w.h.p.
We have 
\[
  \EE W =   
  k
  -
  k
  \bigl(
    1-\frac{1}{n}
  \bigr)^m
  \simeq
  k
  \bigl(
    1-e^{-1/2}
  \bigr)
  <
  \frac{k}{2}
  \simeq
  \EE R_A
  =
  \EE \widetilde{R}_A.
\]
Therefore, by the Azuma inequality 
\[
    \PP\Bigl(
        W 
        \le 
        k
        \bigl(
          1-e^{-1/2}
        \bigr)
        +
        t
    \Bigr)
    \ge 
    1-e^{-ct^2/m}
    \simeq 
    1-e^{-2ct^2/n}
\]
for any $t\ge 0$ and some universal constant $c>0$.
Choose any $l\simeq \frac{k}{2}\bigl(1-e^{-1/2}+\frac{1}{2}\bigr)$ so that $\EE W \le l \le \EE R_A$.
Then the above Azuma inequality implies that $W\le l$ happens with probability at least $1-\exp(-2c'\frac{k^2}{n})$ for some universal constant $c'>0$.
Choose also $k= \Theta(n^{1/2+\varepsilon})$, for some $\varepsilon\in(0,0.5]$ (recall we also assume $k\le n/2$) so that 
\[
  \EE\vert Z'\vert{\bf 1}_{\{W > l\}}
  \le 
  2m 
  \PP(W > l)
  \lesssim
  2m 
  e^{-2c'n^{2\varepsilon}}
  = 
  o(1),
\]
where $x_n\lesssim y_n$ if $x_n \le Cy_n$ for some universal constant $C>0$, whence 
\[
  \EE Z'
  \simeq 
  \EE \widetilde{R}_A
  -
  \EE \widetilde{R}_B 
  =
  \frac{k}{4}.  
\]
On the other hand,
\[
  \EE Z \le 
  l - \EE R_B
  \simeq 
  \frac{k}{2}\bigl(1-e^{-1/2}+\frac{1}{2}\bigr)
  -
  \frac{k}{4}
   =
   \frac{k}{2}\bigl(1-e^{-1/2}\bigr)
\]
and thus 
\[   
    \EE Z' - \EE Z 
    \ge 
    \EE Z' -l + \EE R_B
    \simeq
    \frac{k}{2}(e^{-1/2}-\frac{1}{2})
    \ge
    0.05k.
\]
Consequently, the bound obtained by Tolstikhin--Blanchard--Kloft,~\cite[Theorem 2]{tolstikhin2014localized},
\[
    \forall\;t\ge0\qquad
    \PP(Z \ge \EE Z' + t)
    \le 
    \exp\Bigl(
        -
        t
        \log\bigl( 
            1+\frac{t}{v}
        \bigr)
        +
        t 
        -
        v
        \log\bigl( 
            1+\frac{t}{v}
        \bigr)
    \Bigr),
\]
does not provide a deviation estimate above $\EE Z+t$ for any parameter $t\in [0, 0.05k]$.
On the other hand, the bound from our Theorem~\ref{T:Z-Bennett-general} yields
\[
    \forall\; t\ge 0
    \qquad 
    \PP(Z \ge \EE Z + t)
    \le 
    2\exp\Bigl(
        -\frac{t}{C_1}
        \log\Bigl(
            1 + \frac{t}{C_2\EE \widetilde{\Sigma}^2}
        \Bigr)
    \Bigr),
\]
which for $t=\alpha k$, recalling that $\EE\widetilde{\Sigma}^2=\Theta(k)$, reads
\[
    \PP(Z \ge \EE Z + \alpha k)
    \le 
    2\exp\bigl(
        -c''\alpha k
    \bigr),
\]
for some absolute positive constant $c''>0$.
Finally, we note that the latter inequality can be also obtained from the Talagrand convex distance inequality on the symmetric group~\cite{MR1419006}.
\end{document}